\documentclass[11pt]{article}

\usepackage{amssymb,amsmath}
\usepackage{amsthm}
\usepackage{hyperref}
\usepackage{mathrsfs}
\usepackage{textcomp}

\usepackage{verbatim}

\usepackage{sectsty}

\makeatletter

\newdimen\bibspace
\renewenvironment{thebibliography}[1]{%
 \section*{\refname 
       \@mkboth{\MakeUppercase\refname}{\MakeUppercase\refname}}%
     \list{\@biblabel{\@arabic\c@enumiv}}%
          {\settowidth\labelwidth{\@biblabel{#1}}%
           \leftmargin\labelwidth
           \advance\leftmargin\labelsep
           \itemsep\bibspace
           \parsep\z@skip     %
           \@openbib@code
           \usecounter{enumiv}%
           \let\p@enumiv\@empty
           \renewcommand\theenumiv{\@arabic\c@enumiv}}%
     \sloppy\clubpenalty4000\widowpenalty4000%
     \sfcode`\.\@m}
    {\def\@noitemerr
      {\@latex@warning{Empty `thebibliography' environment}}%
     \endlist}

\makeatother

\makeatletter

\newtheorem{thm}{Theorem}[section]

\newtheorem{prop}[thm]{Proposition}
\newtheorem{defn}[thm]{Definition}
\newtheorem{ex}[thm]{Example}
\newtheorem{cor}[thm]{Corollary}
\newtheorem{rem}[thm]{Remark}

\def\XXint#1#2#3{{\setbox0=\hbox{$#1{#2#3}{\int}$}
  \vcenter{\hbox{$#2#3$}}\kern-.5\wd0}}

\newcommand{\al}{\alpha}                \newcommand{\lda}{\lambda}
\newcommand{\om}{\Omega}                \newcommand{\pa}{\partial}
\newcommand{\va}{\varepsilon}           \newcommand{\ud}{\mathrm{d}}
\newcommand{\be}{\begin{equation}}      \newcommand{\ee}{\end{equation}}

\newcommand{\R}{\mathbb{R}}

\begin{document}

\title{\textbf{A Liouville theorem for solutions of degenerate Monge-Amp\`ere equations}
\bigskip}

\author{\medskip Tianling Jin \  and \
Jingang Xiong}


\date{}

\maketitle

\begin{abstract}
In this paper, we give a new proof of a celebrated theorem of J\"orgens which states that every classical convex solution of
\[
\det\nabla^2 u (x)=1\quad \mbox{in } \mathbb{R}^2
\]
has to be a second order polynomial. Our arguments do not use complex analysis, and can be applied to establish such Liouville type theorems for solutions of a class of degenerate Monge-Amp\`ere equations.
We prove that every convex generalized (or Alexandrov) solution of
\[
\det \nabla^2 u(x_1,x_2)=|x_1|^{\alpha} \quad \mbox{in } \mathbb{R}^2,
\]
where $\alpha>-1$, has to be
\[
u(x_1,x_2)=
\frac{a}{(\alpha+2)(\alpha+1)}|x_1|^{2+\alpha}+\frac{a b^2}{2}x_1^2 +bx_1x_2+ \frac{1}{2a} x_2^2+\ell(x_1,x_2)
\]
for some constants $a>0$, $b$ and a linear function $\ell(x_1,x_2)$.

This work is motivated by the Weyl problem with nonnegative Gauss curvature.

\end{abstract}

\section{Introduction}
A celebrated theorem of J\"orgens states that every entire classical convex solution of
\be\label{eq:jorgens}
\det\nabla^2 u (x)=1
\ee
in $\R^2$ has to be a second order polynomial. This theorem was first proved by J\"orgens \cite{J} using complex analysis methods.
An elementary and simpler proof, which also uses complex analysis, was later
given by Nitsche \cite{N}, where Bernstein theorem for two dimensional minimal
surfaces is established as a corollary.
J\"orgens' theorem was extended to smooth convex solutions in higher dimensions by Calabi \cite{C} for
dimension $\le 5$ and by Pogorelov \cite{P} for all dimensions.  Another proof was given by Cheng and Yau \cite{CY}  along the lines of affine
geometry. Note that each local
generalized solution of \eqref{eq:jorgens} in dimension two is smooth, but this is
false in dimension $\geq 3$. Caffarelli \cite{Caffarelli} established the J\"orgens-Calabi-Pogorelov theorem for generalized solutions (or viscosity solutions). Trudinger-Wang \cite{TW2} proved that the only convex open subset $\om$ of $\R^n$ which admits a convex $C^2$ solution of \eqref{eq:jorgens} in $\om$ with $\lim_{x\to\pa\om}u(x)=\infty$ is $\om=\R^n$. Caffarelli-Li \cite{CL} established the asymptotical behaviors of viscosity solutions of \eqref{eq:jorgens} outside of a bounded convex subset of $\R^n$ for $n\ge 2$ (the case $n=2$ was studied before in Ferrer-Mart\'inez-Mil\'an \cite{FMM2, FMM} using complex analysis), from which the J\"orgens-Calabi-Pogorelov theorem follows.

In this paper, we provide a new proof of this J\"orgens' theorem. Our arguments do not use complex analysis. This allows us to establish such Liouville type theorems for solutions of a class of degenerate Monge-Amp\`ere equations. More precisely, we classify entire convex solutions of the degenerate Monge-Amp\`ere equations
\be\label{eq:first}
\det \nabla^2 u(x_1,x_2)=|x_1|^{\al} \quad \mbox{in } \R^2,
\ee
where $\al>-1$. The equation \eqref{eq:first} appears, for instance, as a blowup limiting equation of
\be\label{eq:full}
\det\nabla^2 u(x_1,x_2)=(x_1^2+x_2^2)^{\al/2}
\ee
in Daskalopoulos-Savin \cite{DS} in the study of the Weyl problem with nonnegative Gauss curvature.

In 1916, Weyl \cite{W} posed the following problem: Given a Riemannian metric $g$ on the $2$-dimensional sphere $\mathbb{S}^2$ whose Gauss curvature is positive everywhere, does there exist a global $C^2$ isometric embedding $X: (\mathbb{S}^2, g)\to (\R^3, \ud s^2)$, where $\ud s^2$ is the standard flat metric on $\R^3$?

Lewy \cite{Lewy} solved the problem in the case that $g$ is real analytic. In 1953, Nirenberg \cite{Nirenberg} gave a solution to this problem under the regularity assumption that $g$ has continuous fourth order derivatives. The result was later extended to the case that $g$ has continuous third order derivatives by Heinz \cite{Heinz}. An entirely different approach was taken independently by Alexandrov and Pogorelov; see \cite{Alex, P1, P2}.

There are also work (see \cite{Iaia, GuanLi1, HZ, DS}) which study the problem with nonnegative Gauss curvature. Guan-Li \cite{GuanLi1} showed that for any $C^4$ metric on $\mathbb{S}^2$ with nonnegative Gauss curvature, there always exists a global $C^{1,1}$ isometric embedding into $(\R^3, \ud s^2)$; see also Hong-Zuily \cite{HZ} for a different approach to this $C^{1,1}$ embedding result. Guan and Li asked there that whether the $C^{1,1}$ isometric embeddings can be improved to be $C^{2,\gamma}$ or even $C^{2,1}$. The problem can be reduced to regularity properties of solutions of a Monge-Amp\`ere equation that becomes degenerate at the points where the Gauss curvature vanishes. If the Gauss curvature of $g$ only has one nondegenerate zero, the regularity of the isometric embedding amounts to studying the regularity of solutions of \eqref{eq:full} near the origin for $\alpha=2$, and it has been proved in Daskalopoulos-Savin \cite{DS} that the solutions of \eqref{eq:full} are $C^{2,\gamma}$ near the origin for $\al>0$.

A comprehensive introduction to the Weyl problem and related ones can be found in the monograph Han-Hong \cite{HanHong}.

The main result of this paper is the following:

\begin{thm}\label{thm:main}
Let $u$ be a convex generalized (or Alexandrov) solution of \eqref{eq:first} with $\al>-1$.
Then there exist some constants $a>0$, $b$ and a linear function $\ell(x_1,x_2)$ such that
\[
u(x_1,x_2)=\frac{a}{(\al+2)(\al+1)}|x_1|^{2+\al}+\frac{a b^2}{2}x_1^2 +bx_1x_2+ \frac{1}{2a} x_2^2+\ell(x_1,x_2).
\]
\end{thm}
Recall that every generalized solution of \eqref{eq:jorgens} in an open subset of $\R^2$ is strictly convex (and thus, smooth). However, this is not the case for generalized (or even classical) solutions of $\det\nabla^2 u=|x_1|^\al$ when $\al>0$; see Example \ref{example}. And it follows from \cite{Caffarelli1} that the generalized solutions of such equations with homogenous boundary condition are strictly convex.

The paper is organized as follows. To illustrate our method, in Section \ref{alpha=0} we first present another proof of J\"orgens' theorem, which only makes use of a few properties of harmonic functions. Those properties also hold in general for solutions of elliptic or even certain degenerate elliptic equations, such as a Grushin type equation shown in Section \ref{sec:grushin} that the partial Legendre transform of $u$ satisfies. In Section \ref{sec:proof}, we show that entire solutions of \eqref{eq:first} are strictly convex and prove Theorem \ref{thm:main}.

\bigskip

\noindent\textbf{Acknowledgements:} Both authors thank Professor YanYan Li
for valuable suggestions and constant encouragement. The second author was supported in part by the First Class Postdoctoral
Science Foundation of China (No. 2012M520002).

\section{A new proof of J\"orgens' theorem}\label{alpha=0}

\begin{proof}[Proof of Theorem \ref{thm:main} when $\al=0$]
First of all, we know that $u$ is smooth. Define $T:\R^2\to \R^2$ by
\be\label{eq:the partial legendre t}
T(x_1,x_2)=(x_1, \nabla_{x_2} u(x))=:(p_1,p_2).
\ee
Clearly, $T$ is injective.
Recall that the partial Legendre transform $u^*(p)$ is defined as
\[
u^*(p)=x_2\nabla_{x_2} u(x)-u(x).
\]
Then
\begin{itemize}
\item $u^*$ is concave w.r.t. $p_1$ and convex w.r.t. $p_2$;

\item $(u^*)^*=u$;

\item $\Delta u^*=0$ in $T(\R^2)$.

\end{itemize}

\emph{Step 1}: Prove the theorem under the assumption $T(\R^2)=\R^2$.

\medskip
For simplicity, we will denote $\nabla_{x_i}u(x), \nabla_{p_i}u^*(p)$ as $u_i(x), u^*_i(p)$ respectively throughout the paper if there is no possibility of confusion. Since $u^*$ is convex w.r.t. $p_2$, we have
\[
u^*_{22}\geq 0\quad \mbox{and}\quad\Delta u^*_{22}=0 \quad\mbox{in}\quad \R^2.
\]
It follows from Liouville theorem for entire nonnegative harmonic functions that $u^*_{22}=a\geq 0$ for
some constant $a$. By the equation of $u^*$, we have $u^*_{11}=-a$.  Hence,
\[
u^*=(-p_1^2+p_2^2)a/2+bp_1p_2+\ell(p_1,p_2)
\]
for some constant $b$ and linear
function $\ell$. Since $u=(u^*)^*$, $a>0$ and we are done.

\medskip

\emph{Step 2}: Prove $T(\R^2)=\R^2$.

\medskip

We prove it by contradiction. Suppose that there exists $\bar{x}_1$ such that \[\lim\limits_{x_2\to +\infty}  u_2(\bar{x}_1,x_2):=\beta<+\infty.\] Claim: for any $x_1\in \R$, \[\lim\limits_{x_2\to +\infty} u_2(x_1,x_2)=\beta.\]

Indeed, by the convexity of $u$, for $t>0$
\[
u(\bar{x}_1,0)+t\beta\geq u(\bar{x}_1, t)\geq u(x_1, x_2)+u_1(x)(\bar{x}_1-x_1)+ u_2(x)(t-x_2),
\]
namely,
\[
u_2(x)(1-x_2/t) \leq \beta + \frac{1}{t}\{u(\bar{x}_1,0)- u(x_1, x_2)-u_1(x)(\bar{x}_1-x_1)\}.
\]
Sending $t\to \infty$, we have $u_2(x_1,x_2)\leq \beta$. Hence, $\lim\limits_{x_2\to +\infty}  u_2(x_1,x_2)\leq \beta$.
Repeating this argument with $x_1$ and $\bar{x}_1$ exchanged,  we would see that  $\lim\limits_{x_2\to +\infty} u_2(x_1,x_2)\geq \beta$.

Without loss of generality, we assume that $\beta=1$. Therefore,
\[
T(\R^2)=(-\infty,+\infty)\times (\beta_0,1)
\]
for some $-\infty\leq \beta_0<1$.
Since $T$ is one-to-one and $u^*_{2}(p_1,p_2)=x_2$, we have
\[
\lim_{p_2\to 1^-}u^*_{2}(p_1,p_2)=+\infty,
\]
 i.e., for any $C>2$, there exists $\va$ (may depend on $\bar p_1$ which is arbitrarily fixed) such that $u^*_{2}(\bar p_1,p_2)\ge C$ for every $p_2\ge 1-\va$. By continuity of $u^*_{2}$, $u^*_{2}(p_1,1-\va)\ge C-1$ for $p_1\in(\bar p_1-\delta, \bar p_1+\delta)$ for some small $\delta$. Since $u^*_{2}$ is monotone increasing in $p_2$, we have $u^*_{2}(p_1,p_2)\ge C-1$ in $(\bar p_1-\delta, \bar p_1+\delta)\times (1-\va,1)$. This shows that
 \[
 \lim_{(p_1,p_2)\to (\bar p_1,1)}u^*_2(p_1,p_2)=+\infty
 \]
 for any $\bar p_1\in \R$, and in particular, $u^*_2$ is positive near the point $(2,1)$. Without loss of generality, we may assume that $u^*_2$ is positive in $[1,3]\times [0,1)$. For any $C>0$ large, we let
 \[
 v(p_1,p_2):=u^*_2(p_1,p_2)-Cp_2(p_1-1)(3-p_1)-\frac{C}{3}p_2^3+\frac C3.
 \]
 Since $\Delta u^*_2=0$, it follows that $\Delta v=0$. By the maximum principle, $v\geq 0$ in $[1,3]\times [0,1)$. In particular, $v(2,\bar p_2)\ge 0$ where $\bar p_2\in (0,1)$ is chosen such that
 \[
 \bar p_2+\bar p_2^3/3-1/3=1/2.
 \] Hence, $u^*_2(2,\bar p_2)\ge C/2$ for all $C>0$, which is a contradiction.
\end{proof}

\section{Homogenous Grushin type equations}\label{sec:grushin}

Let $\om$ be a bounded domain in $\R^n$ with $C^2$ boundary $\pa \om$ such that $\om\cap\{x|x_1=0\}\neq\emptyset$. Consider
\be\label{eq:grushin}
Lu:=u_{x_1x_1}+|x_1|^\al u_{x_2x_2}=0\quad \mbox{in } \om,
\ee
where $\al>-1$. We will see later that the partial Legendre transform of solutions of \eqref{eq:first} satisfies \eqref{eq:grushin}. Also, \eqref{eq:grushin} appears in \cite{CS} in extension formulations for fractional Laplacian operators.
\begin{defn}
We say a function  $u$ is a strong
solution of \eqref{eq:grushin} if $u\in C^1(\om)\cap C^2(\om \setminus\{x_1=0\}) $ and satisfies
\[
Lu=0\quad \mbox{in } \om\setminus\{x_1=0\}.
\]
\end{defn}

In this following, we will see that our definition of strong solution coincides with the classical strong solutions.
Indeed, $u\in W^{2,p}_{loc}$ for any $1\leq p<-\frac{1}{\al}$ if $\al\in (-1,0)$, and $u$ is $C^{2,\delta}$ if $\al\geq 0$.
We have to be careful if we want to study continuous viscosity solutions of \eqref{eq:grushin} which may not have uniqueness property, see Remark 4.3 in \cite{CS}. However, $L^p$-viscosity solutions of certain elliptic equations with coefficients deteriorating along some lower dimensional manifolds would be such strong solutions, see, e.g., \cite{X}. The following proposition is in the same spirit of Lemma 4.2 in \cite{CS}. For regularity properties of solutions of a more general class of quasilinear degenerate elliptic equations we refer to \cite{Guan}.

\begin{prop} \label{prop:existence}
For any $g\in C(\pa \om)$, there exists a unique strong solution $u$ of \eqref{eq:grushin}
with $u\in C(\overline{\om})$ and $u=g$ on $\pa \om$. Furthermore, we have
\be\label{eq:mp}
\max_{\overline \om} u\leq \max_{\pa \om}g, \quad \min_{\overline \om} u\geq \min_{\pa \om}g,
\ee
and, for any $\om'\subset \subset \om$ and $k\in \mathbb{N}$,
\be\label{eq:estimate}
\sum_{l=0}^k\|\nabla^l_{x_2}u\|_{C^1(\om')}\leq C\|g\|_{C^0(\pa \om)},
\ee
where $C>0$ depends only on $n, \al, k, dist(\om',\pa \om)$.
\end{prop}

\begin{proof}
\emph{Uniqueness.} Clearly, the uniqueness would follow from \eqref{eq:mp}. The proof of uniqueness in Lemma 4.2 in \cite{CS} can be applied to obtain \eqref{eq:mp} and we include it for completeness. Let $u$ be a strong solution of \eqref{eq:grushin} with $u\in C(\overline{\om})$ and $u=g$ on $\pa \om$. Let $v=u-\max_{\pa \om}g+\va|x_1|$, where $\va$ is small. Suppose $v$ has an interior maximum point $\bar x$ in $\om$. Then $\bar x_1=0$, since otherwise $v$ satisfies an elliptic equation near $\bar x$ which does not allow an interior maximum point. On the other hand, if $\bar x_1=0$, then $\bar x$ can not be a maximum point of $v$ since $\pa_+v(\bar x)>\pa_-v(\bar x)$. Therefore, the maximum of $v$ is achieved on $\pa \om$, i.e. $u-\max_{\pa \om}g+\va|x_1|\leq \va\mbox{ diam}(\om)$. Sending $\va\to 0$, we obtain $\max_{\overline \om} u\leq \max_{\pa \om}g$. Similarly, we can show that $\min_{\overline \om} u\geq \min_{\pa \om}g$.

\emph{Existence.} For $\va>0$ sufficiently small, let $0<\eta_\va(x_1)\in C^\infty(-\infty,\infty)$ such that
\[
\eta_\va(x_1)=|x_1|^{\al}\quad\mbox{for} \quad |x_1|>2\va;\quad  \eta_\va(x_1)=\va^{\al}\quad\mbox{for} \quad  |x_1|\leq \va.
\]
By the standard linear elliptic equation theory, there exists a unique solution $u^\va\in C(\overline \om)\cap  C^\infty(\om)$ of
\be\label{eq:app}
L_\va u^\va:=u^\va_{x_1x_1}+\eta_\va u^\va_{x_2x_2}=0\quad \mbox{in } \om,
\ee
and $u^\va=g$ on $\om$. By the maximum principle, we have $\sup_{\om}|u^\va|\leq \sup_{\pa \om} |g|$. We will establish proper uniform norms of $u^\va$ and obtain the desired solution by sending $\va\to 0$.

Our proof of this part is different from \cite{CS} which uses Caffarelli-Guti\'errez's Harnack inequality \cite{CG} to obtain uniform interior H\"older norms of those approximating solutions. Instead, we  establish an interior bound of $u^\va_{x_2}$ first, as in Daskalopoulos-Savin \cite{DS}.
In view of the standard uniformly elliptic equation theory, we only need to concern about the area near $\{x_1=0\}$.
Suppose that $0\in \om$ and $B_\tau \subset \om$ for some small $\tau>0$. We shall show
that $\|u^\va_{x_2}\|_{L^\infty(B_{\tau/2})}\leq C$ for some $C$ independent of $\va$.

We claim that there exists a large universal constant $\beta$ such that
\be\label{eq:diff ineq}
L_\va(\beta (u^\va)^2+ (\varphi u_{x_2}^\va)^2)\geq 0\quad \mbox{in }\om,
\ee
where $\varphi$ is some cutoff function in $B_\tau$ satisfying $\varphi=1$ in $B_{\tau/2}$,
$\varphi=0$ in $\om\setminus B_\tau$,  and $\varphi_{x_1}=0$ for all $|x_1|\leq \tau/4$.

Indeed, a simple computation yields
\[
L_\va(u^\va)^2 =2((u^\va_{x_1})^2+\eta_\va (u^\va_{x_2})^2)
\]
and
\[
\begin{split}
L_\va(\varphi u_{x_2}^\va)^2&=L_\va\varphi^2 (u_{x_2}^\va)^2+\varphi^2 L_\va(u_{x_2}^\va)^2+2(\varphi^2)_{x_1}  ((u_{x_2}^\va)^2)_{x_1}+2\eta_\va(\varphi^2)_{x_2}  ((u_{x_2}^\va)^2)_{x_2}\\&
=L_\va\varphi^2 (u_{x_2}^\va)^2 +2\varphi^2((u_{x_2x_1}^\va)^2+\eta_\va(u_{x_2x_2}^\va)^2)+8(\varphi_{x_1} u_{x_2}^\va)(\varphi u_{x_2x_1}^\va)\\&
\quad + 8\eta_\va(\varphi_{x_2}u^\va_{x_2})(\varphi u^\va_{x_2x_2}).
\end{split}
\]
Hence,
\[
\begin{split}
L_\va(\beta (u^\va)^2+ (\varphi u_{x_2}^\va)^2)\geq & 2\beta \eta_\va (u^\va_{x_2})^2+ 2\varphi^2((u_{x_2x_1}^\va)^2+\eta_\va(u_{x_2x_2}^\va)^2)\\&
+L_\va\varphi^2 (u_{x_2}^\va)^2+8(\varphi_{x_1} u_{x_2}^\va)(\varphi u_{x_2x_1}^\va)
+ 8\eta_\va(\varphi_{x_2}u^\va_{x_2})(\varphi u^\va_{x_2x_2}).
\end{split}
\]
By the Cauchy inequality and the facts
\[
L_\va(\varphi^2)\geq -C_1\eta_\va, \quad |\varphi_{x_1} u^\va_{x_2} |\leq C_1\eta_\va|u^\va_{x_2}|,
\]
the claim follows for large $\beta$ independent of $\va$.

By \eqref{eq:diff ineq} and the maximum principle, we have
\[
\sup_{B_{\tau/2}}|u^\va_{x_2}|\leq \beta^{1/2} \sup_{\om}|u^\va|.
\]

Since $Lu^\va_{x_2}=0$, the same arguments can be applied inductively to show that $\pa^k u^\va/\pa{x_2^k}$
are bounded in the interior of $\om$ for any $k\in \mathbb{Z}^+$. Since  $|u^\va_{x_2x_2}|\leq C$ for
some $C$ independent of $\va$ and $u^\va_{x_1x_1}+\eta_\va u^\va_{x_2x_2}=0$, we have
\[
|u^\va_{x_1}|\leq C \int_{-1}^{1}\eta_\va(x_1)\,\ud x_1+C,
\]
where we used the fact that $u^\va_{x_1}$ is bounded uniformly for $B_{3\tau/4}\cap\{x||x_1|\geq \tau/4\}$. Since $\al>-1$,
the integral  $ \int_{-1}^{1}\eta_\va(x_1)\,\ud x_1$ can be bounded independent of $\va$. The same arguments
would show that $u^\va_{x_1x_2}$ and $u^\va_{x_1x_2x_2}$ are bounded as well.

For $\al\in (-1,0)$ and any point $\bar{x}=(\bar{x}_1,\bar{x}_2)\in B_{\tau/4}$, by the Taylor's formula we have
\[
\begin{split}
&u^\va(x_1,\bar{x}_2)\\&=u^\va(\bar{x}_1,\bar{x}_2)+u^\va_{x_1}(\bar{x}_1,\bar{x}_2)(x_1-\bar{x}_1)+
(x_1-\bar{x}_1)^2\int_0^1(1-\lda)u^\va_{x_1x_1}(\xi_\lda,\bar{x}_2)\,\ud \lda \\ &
=u^\va(\bar{x}_1,\bar{x}_2)+u^\va_{x_1}(\bar{x}_1,\bar{x}_2)(x_1-\bar{x}_1)
-(x_1-\bar{x}_1)^2\int_0^1(1-\lda)u^\va_{x_2x_2}(\xi_\lda,\bar{x}_2)\eta(\xi_\lda)\,\ud \lda\\&
=u^\va(\bar{x}_1,\bar{x}_2)+u^\va_{x_1}(\bar{x}_1,\bar{x}_2)(x_1-\bar{x}_1)
-u^\va_{x_2x_2}(\bar{x}_1,\bar{x}_2)(x_1-\bar{x}_1)^2\int_0^1(1-\lda)\eta(\xi_\lda)\,\ud \lda\\&
\quad +O(|x_1-\bar{x}_1|^3\int_0^1\eta(\xi_\lda)\,\ud \lda),
\end{split}
\]
where $\xi_\lda=\bar{x}_1+\lda(x_1-\bar{x}_1)$.  One should note that $\int_0^1\eta(\xi_\lda)\,\ud \lda\leq C|x_1-\bar{x}_1|^{\al}$
for some constant $C>0$ independent of $\va$.
Making use of Taylor's formula again, we have
\[
\begin{split}
u^{\va}(x_1,x_2)=&u^\va(x_1,\bar{x}_2)+u^\va_{x_2}(x_1,\bar{x}_2)(x_2-\bar{x}_2)+\frac12 u^\va_{x_2x_2}(\bar{x}_1,\bar{x}_2)(x_2-\bar{x}_2)^2\\
&+O(|x_2-\bar{x}_2|^3+|(x_1-\bar{x}_1)(x_2-\bar{x}_2)^2|),
\end{split}
\]
and
\[
u^\va_{x_2}(x_1,\bar{x}_2)= u^\va_{x_2}(\bar{x}_1,\bar{x}_2)+u^{\va}_{x_1x_2}(\bar{x}_1,\bar{x}_2)(x_1-\bar{x}_1)+O(|x_1-\bar{x}_1|^{2+\al}).
\]
Therefore,
\[
| u^{\va}(x_1,x_2) -u^\va(x_1,\bar{x}_2)-u^\va_{x_1}(\bar{x}_1,\bar{x}_2)(x_1-\bar{x}_1)-
u^\va_{x_2}(\bar{x}_1,\bar{x}_2)(x_2-\bar{x}_2)|\leq C|x-\bar{x}|^{2+\al}.
\]
By the arbitrary choice of $\bar{x}$, we conclude that
\be\label{eq:al small}
\|u^\va\|_{C^{1,1+\al}(B_{\tau/4})}\leq C.
\ee
The same argument is also applicable to $\al\ge 0$, and one can conclude that
\be\label{eq:al large}
\|u^\va\|_{C^{2,\delta}(B_{\tau/4})}\leq C
\ee
for some $\delta>0$ depending only on $\al$.

By passing to a subsequence, we obtain a strong solution $u$ of \eqref{eq:grushin} and $u$ satisfies \eqref{eq:estimate}.
\end{proof}

\begin{rem}\label{rem:regu}
From the proof of Proposition \ref{prop:existence}, we see that:
\begin{itemize}
\item If $\al\in(-1,0)$, $u\in C^{1,1+\al}_{loc}(\om)$;
\item If $\al\geq 0$, $u\in C^{2,\delta}_{loc}(\om)$ for some $\delta>0$ depending only on $\al$.
\end{itemize}
\end{rem}

\bigskip

Let
\be\label{eq:standard}
\phi(x_1,x_2)=|x_1|^{2+\al}+x_2^2 \quad \mbox{in }\R^2.
\ee
Then
\[
\nabla^2 \phi=\left(
                \begin{array}{cc}
                  (2+\al)(1+\al)|x_1|^\al & 0 \\
                  0 & 2 \\
                \end{array}
              \right),
\]
and
\[
(\nabla^2 \phi)^{1/2}=\left(
                \begin{array}{cc}
                  \sqrt{(2+\al)(1+\al)}|x_1|^{\al/2} & 0 \\
                  0 & \sqrt{2} \\
                \end{array}
              \right).
\]
Hence, $\det \nabla^2 \phi= c(\al)|x_1|^\al$, where $c(\al)=2(\al+2)(\al+1)>0$.
For any $x\in \R^2$ and $t>0$, denote
\[
S(x,t)=S_\phi(x,t)=\{y\in \R^2| \phi(y)<\ell(y)+t\},
\]
where $\ell(y)$ is the support plane of $\phi$ at $(x,\phi(x))$. It is direct to verify

\medskip

\noindent \textbf{Condition $\mu_\infty$} \cite{CG}: \emph{For any given $\delta_1\in (0,1)$,
there exists $\delta_2\in (0,1)$ such that, for all sections $S$ and all small subsets $E\subset S$,
\be\label{condition}
\frac{|E|}{|S|}<\delta_2\quad \mbox{implies} \quad \frac{\int_E |x_1|^\al\,\ud x}{\int_S |x_1|^\al\,\ud x}<\delta_1.
\ee
}

Let
\[
A(x_1,x_2)= \left(
                \begin{array}{cc}
                  |x_1|^{-\al} & 0 \\
                  0 & 1 \\
                \end{array}
              \right).
\]
Clearly,
\[
B:=(\nabla^2 \phi)^{1/2}A(\nabla^2 \phi)^{1/2}=\left(
                \begin{array}{cc}
                 (2+\al)(1+\al) & 0 \\
                  0 & 2 \\
                \end{array}
              \right),
\]
which is positive definite if $\al >-1$. Therefore, we can apply Caffarelli-Guti\'errez's Harnack inequality \cite{CG} to obtain the following proposition.

\begin{prop}\label{prop:harnack} Let $u\geq 0$ be a strong solution of
\[
Lu=0 \quad \mbox{in } S(x_0,2),
\]
where $x_0$ is an arbitrary point in $\R^2$.
Then there exists a positive constant $\beta$ depending only on $\al$ such that
\[
\sup_{S(x_0,1)} u\leq \beta \inf_{S(x_0,1)} u.
\]
\end{prop}

\begin{cor}\label{cor:holder}
Let $u$ be a strong solution of
\[
Lu=0 \quad \mbox{in } S(0,2).
\]
Then there exist constants $C>0$ and $\gamma\in (0,1)$
depending only on $\al$ such that
\[
\|u\|_{C^\gamma(S(0,1))}\leq C\|u\|_{L^\infty(S(0,2))}.
\]
\end{cor}

\begin{thm}\label{thm:1} Let $u$ be a nonnegative strong solution of
\be\label{eq: to liou}
Lu=0\quad \mbox{in }\R^2.
\ee
Then $u$ is a constant in $\R^2$.
\end{thm}

\begin{proof}
Consider the scaling $u_r=\frac{1}{r}u(r^{1/(2+\al)}x_1, r^{1/2}x_2)$ for $r>0$.
Then $u_r$ also satisfies \eqref{eq: to liou}. By Proposition \ref{prop:harnack}, we have
\[
\sup_{S(0,2)} u_r \leq \beta u_r(0).
\]
It follows from Corollary \ref{cor:holder} that
\[
[u_r]_{C^\gamma(S(0,1))}\leq C \beta u_r(0).
\]
For any two distinct points $x,y$ in $\R^2$, we have, for sufficiently large $r$,
\[
\begin{split}
|u(x)-u(y)|&=r|u_r(r^{-1/(2+\al)}x_1, r^{-1/2}x_2)-u_r(r^{-1/(2+\al)}y_1, r^{-1/2}y_2)|\\&
\leq r[u_r]_{C^\gamma(S(0,1))}|r^{-2/(2+\al)}(x_1-y_1)^2+r^{-1}(x_2-y_2)^2|^{\gamma/2}\\ &
\leq C \beta u(0) |r^{-2/(2+\al)}(x_1-y_1)^2+r^{-1}(x_2-y_2)^2|^{\gamma/2}.
\end{split}
\]
Sending $r\to \infty$, we obtain $u(x)=u(y)$. The proof is completed.
\end{proof}

\section{Regularity for solutions of degenerate Monge-Amp\`ere equations}\label{sec:proof}

Define the measure $\mu_\al$ in $\R^2$ as $\ud\mu_\al= |x_1|^{\al}\ud x_1\ud x_2$ for $\al>-1$.
For any bounded open convex set $\om\subset \R^2$, it is clear that the measure $\mu_\al$
has the \emph{doubling property} in $\om$, i.e., there exists a constant $c_\al>0$, depending only on $\al$ and $\om$, such that
for  any $(\bar x_1,\bar x_2)\in \om$ and any ellipsoids $E\subset \R^2$ centered at origin with $(\bar x_1,\bar x_2)+E\in \om$ there holds
\be\label{doubling}
\mu_\al((\bar x_1,\bar x_2)+E)\geq c_\al \mu_\al(((\bar x_1,\bar x_2)+2E)\cap \om).
\ee
Consequently, we have the following theorem.

\begin{thm} \label{thm:reg} Let $\om$ be an open convex set in $\R^2$, and $u$ be the generalized solution of
\[
\det\nabla^2 u(x)=|x_1|^\al \quad \mbox{in } \om,
\]
with $u=0$ on $\pa \om$. Then $u$ is strictly convex in $\om$, $u\in C^{1,\delta}_{loc}(\om)$ for some $\delta>0$ depending only on $\al$.
Furthermore, the partial Legendre transform $u^*$ of $u$ is a strong solution of
\[
Lu^*=0\quad \mbox{in }T(\om),
\]
where the map $T$ is given in \eqref{eq:the partial legendre t}.
\end{thm}
\begin{proof}
The strict convexity and the $C^{1,\delta}$ regularity was proved in \cite{Caffarelli0, Caffarelli1}. Hence, $T$ is continuous and one-to-one, and thus, $T(\om)$ is open. Let $u_k\in C(\overline \om)\cap C^\infty(\om)$
be the solution of
\be\label{eq:amae}
\det\nabla^2 u_k=\eta_{1/k}(x_1)\quad \mbox{in }\om
\ee
with $u_k=0$ on $\pa \om$, where $\eta_{1/k}(x_1)$ is the same as the one in the proof of Proposition \ref{prop:existence} with $\va=1/k$. Let
\[
T_{k}: \om \to \R^2, \quad (x_1,x_2)\mapsto (x_1, \pa_2u_k(x)),
\]
and $u^*_k$ be the partial Legendre transform of $u_k$.
Then $u^*_k$ satisfies \eqref{eq:app}.
Clearly, up to a subsequence, $u_k\to u$ in $C^{1}_{loc}(\om)$ as $k\to\infty$. Thus, $\lim_{k\to \infty}T_{k}(x)= T(x)$ for any $x\in \om$, and for any $y\in T(\om)$ there exists $\lda$ sufficiently small such that $B_\lda(y)\subset T(\om)\cap T_k(\om)$ for every large $k$.  By the same argument used in proof of Proposition \ref{prop:existence}, we can conclude that
$u^*\in C^1(T(\om))\cap C^2(T(\om)\setminus \{x_1=0\})$ and satisfies $L u^*=0$ in $T(\om)\setminus \{x_1=0\}$.
\end{proof}

\begin{thm} \label{prop:sc2}
Let $u$ be a generalized solution of \eqref{eq:first}. Then $u$ is strictly convex.
\end{thm}

\begin{proof}  By the two dimensional Monge-Amp\`ere equation theory, if $u$ is a generalized solution of
\[
\det \nabla^2 u\geq c_0>0 \quad \mbox{in }\om,
\]
where $\om$ is an open set in $\R^2$,
then $u$ is locally strictly convex in $\om$. Hence, we only need to consider the situation $\al >0$.
After subtracting a supporting plane of $u$ at origin, we may assume that
\[
 u\geq 0\quad \mbox{in }\R^2 \quad \mbox{and } u(0)=0.
\]

\medskip

Claim: There exists a sufficiently large $R>0$ such that
\be \label{eq:stc}
\min_{\pa B_R} u>0.
\ee

\medskip

Indeed, if not, namely, $\min_{\pa B_R} u=0$ for all sufficiently large $R>0$.  The strict convexity of $u$ away from $\{x_1=0\}$ implies $u(Re_2)=0$ or $u(-Re_2)=0$, where $e_2=(0,1)$.
Without loss of generality, we may assume $u(Re_2)=0$. Let
\[
M=\max_{\pa B_1} u>0,
\]
 and $\Delta$ be the triangle generated by the segment $\{(x_1,0)||x_1|\leq 1\}$ and the point $Re_2$. By the convexity of $u$, we have
\[
M\geq u \quad \mbox{in }\Delta.
\]
It is clear that the ellipsoid
\[
E=\{(x_1,x_2):  x_1^2+\frac{1}{R^2}(x_2-R/4)^2=\frac{1}{16}\}
\]
sits in $\Delta$. Let
\[
u_R(y_1,y_2)=\frac{1}{R}u(y_1, R(y_2+1/4)).
\]
We have
\[
\det \nabla^2 u_R(y_1,y_2)=|y_1|^\al \quad \mbox{in }B_{1/4},
\]
and $u_R\leq \frac{M}{R}$ in $B_{1/4}$. Choosing a small constant $\tau>0$, depending only on $\al$, such that
\[
S_\phi(0,\tau)\subset B_{1/4},
\]
where $\phi$ is given in \eqref{eq:standard}. By the comparison principle (see, e.g., \cite{G}),
\[
0\leq u_R\leq \sqrt{c(\al)^{-1}}(\phi-\tau)+\max_{\pa S(0,\tau)}u_R \quad\mbox{in }S_\phi(0,\tau),
\]
where $c(\al)=2(\al+2)(\al+1)$. In particular,
\[0\leq - \sqrt{c(\al)^{-1}}\tau +\max_{\pa S(0,\tau)}u_R \leq -\sqrt{c(\al)^{-1}}\tau+M/R.
\]
That is
\[
R\leq \frac{\sqrt{c(\al)}M}{\tau},
\]
which contradicts to the assumption that $R$ can be arbitrarily large.

Thus, \eqref{eq:stc} holds and we can conclude Theorem \ref{prop:sc2} from Theorem \ref{thm:reg}.
\end{proof}

One might ask if every solution of
\[
\det\nabla^2 u=|x_1|^\al\quad\mbox{in }B_1\subset \R^2
\]
is strictly convex, where $\al>0$. The following example shows that this is not the case.

\begin{ex}\label{example}
It is clear that for every $\al>0$ there always exists a positive convex smooth solution $w$ of the ODE
\be
\begin{split}
\begin{cases}
\frac{\al(\al+2)}{4}w(t)w(t)''-\frac{(\al+2)^2}{4}(w'(t))^2=1,\\
w(0)=1,\\
w'(0)=1,
\end{cases}
\end{split}
\ee
near $t=0$. Then $u=|x_1|^{\frac{\al+2}{2}}w(x_2)$ is a generalized solutions of $\det \nabla^2 u=|x_1|^{\al}$ in a small open set in $\R^2$. But $u$ is not strictly convex (is smooth for certain $\al$, though). By proper scaling and translation we can make the equation holds in $B_1$.

\end{ex}

\begin{proof}[Proof of Theorem \ref{thm:main}] Let $u$ be a generalized solution of \eqref{eq:first}. It follows from Theorem \ref{prop:sc2} that $u$ is strictly convex, and hence $u$ is smooth away from $\{x_1=0\}$. By Theorem \ref{thm:reg}, we know that $u\in C^{1,\delta}_{loc}(\R^2)$ and the partial Legendre transform $u^*$ of $u$ is a strong solution of
\be\label{eq:le}
Lu^*=u^*_{11}+|p_1|^\al u^*_{22}=0\quad \mbox{in } T(\R^2),
\ee
where $u^*_{ii}=u^*_{p_ip_i}$ and $T(x_1,x_2)=(x_1,u_{x_2}(x_1,x_2))=(p_1,p_2)$. Moreover, $T$ is continuous and one-to-one.

Given Theorem \ref{thm:1} and Proposition \ref{prop:harnack}, the rest of the proof is similar to that in Section \ref{alpha=0} for $\al=0$.

\medskip

\emph{Step 1}: Prove the theorem under the assumption: $T(\R^2)=\R^2$.

\medskip

Since $u^*$ is convex with respect to $p_2$, we have that $u^*_{22}\ge 0$. Note that $Lu^*_{22}=0$ in $\R^2$. By Theorem \ref{thm:1}, $u^*_{22}\equiv a$ for some nonnegative constant $a$. By the equation $Lu^*=0$, we have $u^*_{11}=-a|p_1|^\al$. Hence, $u^*_{121}\equiv u^*_{122}\equiv 0$ in $\{p_1>0\}$. Consequently, $u^*_{12}\equiv b$ in $\{p_1>0\}$ for some constant $b$. It follows from calculus that
\be\label{eq:entire function}
u^*=-\frac{a}{(\al+1)(\al+2)}|p_1|^{2+\al}+\frac {a}{2} p_2^2+bp_1p_2+\ell(p_1,p_2)
\ee
for some linear function $\ell$ in $\{p_1>0\}$. The same argument applies to $\{p_1<0\}$. Since $u^*, u^*_{2}\in C^1(\R^2)$, \eqref{eq:entire function} holds for all $p\in \R^2$. Since $u=(u^*)^*$, $a>0$ and we are done.

\medskip

\emph{Step 2}: Prove: $T(\R^2)=\R^2$.

\medskip

We prove it by contradiction. Suppose that there exists $\bar{x}_1$ such that $\lim\limits_{x_2\to \infty}  u_2(\bar{x}_1,x_2):=\beta_2<\infty$.
Then, as in Section \ref{alpha=0}, $\lim\limits_{x_2\to \infty}  u_2(x_1,x_2)=\beta$ for every $x_1\in \R$, and we may assume $\beta=1$. Therefore,  $T(\R^2)=(-\infty,\infty)\times (\beta_0,1)$ for some $-\infty\leq \beta_0<1$. Since $T$ is one-to-one and $u^*_{2}(p_1,p_2)=x_2$, we have $\lim_{p_2\to 1^-}u^*_{2}(p_1,p_2)=\infty$. The same argument in Section \ref{alpha=0} shows that
\[
\lim_{(p_1,p_2)\to (\bar p_1,1)}u^*_2(p_1,p_2)=+\infty
\]
for any $\bar p_1\in \R$.

\medskip

Case 1: $\alpha\ge 0$.

\medskip

Without loss of generality, we may assume that $u^*_2$ is positive in $[1,3]\times [0,1)$. For any $C>0$ large, we let
\[
v(p_1,p_2):=u^*_2(p_1,p_2)-Cp_2(p_1-1)(3-p_1)-\frac{C}{3}p_2^3+\frac C3.
\]
It is direct to check that $Lv<0$ in $[1,3]\times [0,1)$. By the maximum principle, $v\geq 0$ in $[1,3]\times [0,1)$. In particular, $v(2,\bar p_2)\ge 0$ where $\bar p_2\in (0,1)$ is chosen such that
\[
\bar p_2+\bar p_2^3/3-1/3=1/2.
\]
Hence, $u^*_2(2,\bar p_2)\ge C/2$ for all $C>0$, which is a contradiction.

\medskip

Case 2: $\alpha\in (-1,0)$.

\medskip

Without loss of generality, we may assume that $u^*_2$ is positive in $[1/2,1]\times [0,1)$. For any $C>0$ large, we let
\[
v(p_1,p_2):=u^*_2(p_1,p_2)-Cp_2(p_1-1/2)(1-p_1)-\frac{C}{3}p_2^3+\frac C3.
\]
It is direct to check that $Lv<0$ in $[1,3]\times [0,1)$. By the maximum principle, $v\geq 0$ in $[1/2,1]\times [0,1)$. In particular, $v(3/4,\bar p_2)\ge 0$ where $\bar p_2\in (0,1)$ is chosen such that
\[
\bar p_2/16+\bar p_2^3/3-1/3=1/32.
\]
Hence, $u^*_2(3/4,\bar p_2)\ge C/32$ for all $C>0$, which is a contradiction.

\medskip

The proof is completed.
\end{proof}

\bigskip

\noindent T. Jin

\noindent Department of Mathematics, The University of Chicago\\
5734 S. University Avenue, Chicago, IL 60637, USA\\[1mm]
Email: \textsf{tj@math.uchicago.edu}

\medskip

\noindent J. Xiong

\noindent Beijing International Center for Mathematical Research, Peking University\\
Beijing 100871, China\\[1mm]
Email: \textsf{jxiong@math.pku.edu.cn}

\end{document}